\documentclass[12pt]{amsart}
\usepackage{amsmath,amsthm,amsfonts,amssymb,latexsym,enumerate}
\usepackage{showlabels}
\usepackage[pagebackref,colorlinks]{hyperref}
\usepackage[pagewise]{lineno}

\newcommand{\FF}{{\mathbb{F}}}

\newcommand{\bC}{{\mathbf C}}

\newcommand{\bZ}{{\mathbf Z}}

\newcommand{\oo}{\mathcal{O}}

\newcommand{\Syl}{\operatorname{Syl}}

\newcommand{\Irr}{\operatorname{Irr}}

\newcommand{\SL}{\operatorname{SL}}
\newcommand{\PSL}{\operatorname{PSL}}
\newcommand{\PSp}{\operatorname{PSp}}
\newcommand{\PSU}{\operatorname{PSU}}

\newcommand{\Pro}{\operatorname{Pr}}
\newcommand{\dd}{\operatorname{d}}

\newcommand{\diag}{{{\operatorname{diag}}}}

\newcommand{\GL}{\operatorname{GL}}

\newcommand{\Al}{\textup{\textsf{A}}}

\newtheorem{thm}{Theorem}[section]
\newtheorem{lem}[thm]{Lemma}

\newtheorem{pro}[thm]{Proposition}
\newtheorem{cor}[thm]{Corollary}

\theoremstyle{definition}

\theoremstyle{remark}
\newtheorem{remark}[thm]{Remark}

\numberwithin{equation}{section}

\begin{document}


\title[Classes of $\pi$-elements and nilpotent/abelian Hall $\pi$-subgroups]
{Conjugacy classes of $\pi$-elements\\ and nilpotent/abelian Hall
$\pi$-subgroups}

\author[N.\,N. Hung]{Nguyen N.  Hung}
\address{Department of Mathematics, The University of Akron, Akron, OH 44325, USA}
\email{hungnguyen@uakron.edu}
\author[A. Mar\'{o}ti]{Attila Mar\'{o}ti}
\address{Alfr\'ed R\'enyi Institute of Mathematics, Re\'altanoda utca 13-15,
  H-1053, Budapest, Hungary}
\email{maroti@renyi.hu}
\author[J. Mart\'inez]{Juan Mart\'inez}
\address{Departament de Matem\`atiques, Universitat de Val\`encia, 46100
  Burjassot, Val\`encia, Spain}
\email{juanma23@alumni.uv.es}

\thanks{The first and third authors
would like to thank Alexander Moret\'{o} for helpful conversations.
We are grateful to the referee for several suggestions that have
greatly improved the exposition in the paper. Part of this work was
done while the first author was visiting the Vietnam Institute for
Advanced Study in Mathematics (VIASM), and he thanks the VIASM for
financial support and hospitality. The second author would like to
thank the Isaac Newton Institute for Mathematical Sciences,
Cambridge for its support and hospitality during the programme
Groups, representations and applications: new perspectives when work
on this paper was undertaken. The work of the second author has
received funding from the European Research Council (ERC) under the
European Union's Horizon 2020 research and innovation programme
(grant agreement No 741420); he was also supported by the National
Research, Development and Innovation Office (NKFIH) Grant
No.~K138596, No.~K132951 and Grant No.~K138828. }



\begin{abstract}
Let $G$ be a finite group and $\pi$ be a set of primes. We study
finite groups with a large number of conjugacy classes of
$\pi$-elements. In particular, we obtain precise lower bounds for
this number in terms of the $\pi$-part of the order of $G$ to ensure
the existence of a nilpotent or abelian Hall $\pi$-subgroup in $G$.
\end{abstract}

\maketitle


\section{Introduction}

Let $G$ be a finite group. The number $k(G)$ of conjugacy classes of
$G$ is an important and much investigated invariant in group theory.
It is equal to the number of complex irreducible representations of
$G$. The probability $\Pro(G)$ that two uniformly and randomly
chosen elements from $G$ commute is given by $k(G)/|G|$ where $|G|$
denotes the order of $G$. This is called the commuting probability
or the commutativity degree of $G$ and it has a large literature,
see \cite{Gustafson,N,Lescot,GR,Eberhard15} and references therein.
The commuting probability has also been studied for infinite groups,
see \cite{T}.

A starting point of our work is a much cited theorem of Gustafson
\cite{Gustafson} stating that  $\Pro(G)>5/8$ for a finite group $G$
if and only if it is abelian. Let $p$ be the smallest prime divisor
of the order of a finite group $G$. It was observed by Guralnick and
Robinson \cite[Lemma 2]{GR} that if $\Pro(G) > 1/p$, then $G$ is
nilpotent. Moreover, Burness, Guralnick, Moret\'{o} and Navarro
\cite[Lemma 4.2]{Orig} recently showed that if
$\Pro(G)>\frac{p^2+p-1}{p^3}$, then $G$ is abelian. An aim of this
paper is to give a generalization of all three of these results.

Let $\pi$ be a set of primes. A positive integer is called a
$\pi$-number if it is not divisible by any prime outside $\pi$. The
$\pi$-part $n_\pi$ of a positive integer $n$ is the largest
$\pi$-number which divides $n$. An element of a finite group is
called a $\pi$-element if its order is a $\pi$-number. The set of
all $\pi$-elements in a finite group is a union of conjugacy classes
of the group. Let $k_{\pi}(G)$ be the number of conjugacy classes of
$\pi$-elements in a finite group $G$ and let
\[d_\pi(G):=k_{\pi}(G)/|G|_\pi.\] This invariant is always at most $1$ by an
old result of Robinson, see \cite[Lemma 3.5]{GGG}. The main theorem
of the paper \cite{MN} is that if $d_\pi(G)>5/8$ for a finite group
$G$ and a set of primes $\pi$, then $G$ possesses an abelian Hall
$\pi$-subgroup. The following result is a far reaching
generalization of this statement.

\begin{thm}\label{thm:main}
Let $G$ be a finite group and let $\pi$ be a set of primes. Let $p$
be the smallest member of $\pi$. If $\dd_{\pi}(G)> \frac{1}{p}$,
then $G$ has a nilpotent Hall $\pi$-subgroup, whose derived
subgroup has size at most $p$. Moreover, if
$\dd_{\pi}(G)> \frac{p^{2}+p-1}{p^{3}}$, then $G$ has an abelian
Hall $\pi$-subgroup.
\end{thm}

A well-known theorem of Wielandt \cite{Wielandt} states that if a
finite group $G$ contains a nilpotent Hall $\pi$-subgroup for some
set of primes $\pi$ then all Hall $\pi$-subgroups of $G$ are
conjugate and every $\pi$-subgroup of $G$ is contained in a Hall
$\pi$-subgroup. Therefore, the $\pi$-subgroups of a group satisfying
the hypothesis of Theorem \ref{thm:main} behave like Sylow
subgroups.

There are several results in the literature on the existence of
abelian or nilpotent Hall subgroups in finite groups. For example
\cite[Theorem B]{BFMMNSST} states that if $G$ is a finite group and
$\pi$ a set of primes, then $G$ has nilpotent Hall $\pi$-subgroups
if and only if for every pair of distinct primes $p$, $q$ in $\pi$
the class sizes of the $p$-elements of $G$ are not divisible by $q$.

For certain sets $\pi$, Tong-Viet \cite{TV} proved some nice results
on the existence of normal $\pi$-complements in finite groups $G$
under the condition that $\dd_{\pi}(G)$ is large. For example,
\cite[Theorem E]{TV} states that if $p>2$ is the smallest prime in
$\pi$ and $d_\pi(G)>(p+1)/2p$, then $G$ contains not only an abelian
Hall $\pi$-subgroup but also a normal $\pi$-complement. Another is
\cite[Theorem A]{TV}, which states that if $d_2(G)>1/2$ then $G$ has
a normal $2$-complement. We in fact make use of this result to prove
Theorem \ref{thm:main} in the case $2\in\pi$. As a consequence, the proof for this case
does not depend on the classification of finite simple groups. The other case
$2\notin\pi$, however, is more challenging and our proof has to rely on the
classification.

The paper is organized as follows. In Section \ref{sec:preliminary}
we prove some preliminary results on the commuting probability
$\Pro(G)$. In Section \ref{sec:Hall} we prove some basic properties
of the $\pi$-class invariant $\dd_{\pi}(G)$ and, in particular, we show in Theorem
\ref{equivalency} that in order to prove the main result, it
suffices to show the existence of a
nilpotent Hall $\pi$-subgroup under the hypothesis $\dd_{\pi}(G)>
\frac{1}{p}$. We then establish this statement in Section
\ref{sec:reducing}, modulo a result on finite simple groups (Theorem
\ref{teoB}) that will be proved in Section \ref{sec:simple groups}.
Finally, in Section \ref{sec:examples}, we present examples showing
that the converse of Theorem \ref{thm:main} is false and that the
obtained bounds are sharp in general.


\section{Commuting probability}\label{sec:preliminary}

In this section we recall and prove some results about the commuting
probability $\Pro(G)$ that will be needed later.

The first lemma is a generalisation of Gustafson's result   \cite{Gustafson} mentioned
earlier. The inequality part is due to Burness,
Guralnick, Moret\'{o}, and Navarro \cite{Orig}.

\begin{lem}\label{lem1}
    Let $G$ be a finite group and  $p$  the smallest prime dividing
    $|G|$. If $G$ is not abelian, then  $\Pro(G)\leq
    \frac{p^2+p-1}{p^3}$ with equality  if and only if
    $G/\bZ(G)=C_{p}\times C_{p}$.
\end{lem}

\begin{proof}
    The first part of the lemma is \cite[Lemma 4.2]{Orig}. Following its
    proof, we see that the equality $\Pro(G)= \frac{p^2+p-1}{p^3}$ holds
    if and only if $G/\bZ(G)=C_{p}\times C_{p}$ and $|x^G|=p$ for every $x \in G
    \setminus \bZ(G)$. It suffices to prove that if $G/\bZ(G)=C_{p}\times C_{p}$,
    then $|x^G|=p$ for every $x \in G\setminus \bZ(G)$.

    Assume that $G/\bZ(G)=C_{p}\times
    C_{p}$ and let $x \in G \setminus \bZ(G)$. Since $x \in
    \bC_{G}(x)\setminus \bZ(G)$, we have that $\bZ(G) < \bC_{G}(x)$.
    Therefore, $|x^G|=\frac{|G|}{|C_{G}(x)|}$ is a proper divisor of
    $\frac{|G|}{|\bZ(G)|}=p^2$. On the other hand, since $x$ is not
    central, $|x^G|>1$. Thus, $|x^G|=p$, and the claim follows.
\end{proof}

Note that if $G$ is an extraspecial $p$-group of order $p^3$ with
$p$ odd or if $G$ is a dihedral group when $p=2$, then
$G/\bZ(G)=C_{p}\times C_{p}$. Therefore, the bound in Lemma
\ref{lem1} is sharp for all $p$.

We next give a bound for $\Pro(G)$ in terms of the smallest prime
factor of the order of $G$ and the order of its derived subgroup
$G'$.

\begin{lem}\label{bound}
    If $p$ is the smallest prime dividing the order of a finite group $G$, then $$\Pro(G) \leq
    \frac{1 + (p^{2} -1)/|G'|}{p^2}.$$
\end{lem}

\begin{proof}
    Let $\Irr(G)$ denote the set of all irreducible complex characters
    of $G$. We have
    $$|G| = \sum_{\chi \in \mathrm{Irr}(G)} \chi(1)^2 \geq |G/G'| +
    p^{2}(k(G) - |G/G'|),$$ since $\chi(1)$ divides $|G|$ for every
    $\chi\in\Irr(G)$. After dividing both sides of the previous inequality by $|G|$, we obtain $1 \geq
    1/|G'| + p^{2}(\Pro(G) - 1/|G'|)$. This yields
$\Pro(G) \leq
    {(1 + (p^{2} -1)/|G'|)}/{p^2},$ as we claimed.
\end{proof}

\begin{lem}\label{Gprima}
    Let $G$ be a finite group and $p$ the smallest prime dividing
    $|G|$. Suppose that $|G'|\leq p$. Then $G' \leq \bZ(G)$, and thus $G/\bZ(G)$ is abelian. In
    particular, $G$ is nilpotent.
\end{lem}

\begin{proof}
    The case $|G'|=1$ is obvious, so we assume $|G'|=p$. Since $G'$ is
    normal and its order is the smallest prime dividing $|G|$, we deduce
    that $G'$ is central in $G$, and the result follows.
\end{proof}

Next we refine Lemma \ref{lem1}.
It follows from \cite[Lemma 2(xiii)]{GR} of Guralnick and Robinson
that if $\Pro(G)>\frac{1}{p}$, where $p$ is the smallest prime
dividing $|G|$, then $G$ is nilpotent.

\begin{thm}\label{firstpart}
    Let $G$ be a finite group and $p$ the smallest prime dividing $|G|$.
    Then  $\frac{1}{p}<\Pro(G)\leq \frac{p^2+p-1}{p^3}$ if and only if $|G'|=p$. Moreover, in
    such case,
    \[\Pro(G)=\frac{1}{p}+\frac{p-1}{p|G:\bZ(G)|}.\]
\end{thm}

\begin{proof}
By Lemma \ref{lem1} we may assume that $G$ is non-abelian. Assume
that $|G'|>p$. Then $|G'|\geq p+1$ and hence, applying Lemma
\ref{bound}, we have $\Pro(G)\leq \frac{1}{p}$. The only if part is
therefore done.

Conversely, assume that $|G'|=p$. Then $G'\leq \bZ(G)$ by Lemma
\ref{Gprima}. By \cite[Problem 2.13]{Isaacscar}, we have
    $\chi(1)^2=|G:\bZ(G)|$ for every $\chi\in \Irr(G)$ with $\chi(1)>1$. We deduce that
    $$|G| = \sum_{\chi \in \mathrm{Irr}(G)} \chi(1)^2 = |G|/p +
    |G:\bZ(G)|(k(G) - |G|/p),$$
    and it follows that
    $$\Pro(G)=\frac{1}{p}+\frac{p-1}{p|G:\bZ(G)|}>\frac{1}{p},$$ as stated.
\end{proof}

\begin{remark} It is worth noting that if $G/\bZ(G)\cong C_{p} \times C_{p}$, then, by Lemma
    \ref{lem1}, we have $\Pro(G)=\frac{p^2+p-1}{p^3}>\frac{1}{p}$, and
    hence $|G'|=p$ by Theorem \ref{firstpart}.
\end{remark}

Let us denote $$g_p(x):=\frac{1 + (p^{2} -1)/x}{p^2}.$$
We note that the function $g_p(x)$ is decreasing in terms of $x$. Also, $g_{p}(1) = 1$, $g_p(p)
=\frac{p^2+p-1}{p^3}$, and $g_p(p+1)=\frac{1}{p}$. These values of $g_p$, that appear in our main result, explain the relevance of $g_p$.

The next theorem could be compared with a result of Lescot \cite{Lescot} stating that $\Pro(G)=\frac{1}{2}$ if and only if $G$
is isoclinic to the symmetric group $\Sigma_3$.

\begin{thm}
    \label{n(p)}
    Let $G$ be a finite group and  $p$ the smallest prime dividing $|G|$.
    If $|G'|>p$, then
    $$\Pro(G)\leq  \frac{n(p)+p^2-1}{p^2n(p)}\leq \frac{1}{p},$$
    where $n(p)$ denotes the smallest prime larger than
    $p$. Moreover, $\Pro(G)=\frac{1}{p}$ if and only if $p=2$ and $G/\bZ(G)\cong\Sigma_{3}$.
\end{thm}

\begin{proof}
    By Bertrand's postulate, we know that $n(p)<2p\leq p^2$.
    Therefore, if $|G'|>p$ then $|G'|\geq n(p)$ and hence, applying Lemma \ref{bound},
    we have
    $$\Pro(G)\leq g_p(n(p))= \frac{n(p)+p^2-1}{p^2n(p)}.$$
    The second inequality holds because $g_p(n(p))\leq g_p(p+1)=\frac{1}{p}$.

    Suppose that $\Pro(G)=\frac{1}{p}$. This forces $n(p)=p+1$, which implies that
    $p=2$ and $|G'|=3$. We
    claim that $\Pro(G)=\frac{1}{2}$ if and only if
    $G/\bZ(G)=\Sigma_{3}$. Assume first that $G/\bZ(G)=\Sigma_{3}$. Let $q$ be a prime dividing
    $|G|$ and let $Q\in \Syl_{q}(G)$. Since $G/\bZ(G)=\Sigma_{3}$, we
    deduce that $|Q:\bZ(Q)|\leq q$ and hence $Q$ is abelian. It follows
    that $G$ possesses an abelian Sylow $q$-subgroup for every prime $q$
    dividing $|G|$. Thus, by \cite[Lemma 2(xiii)]{GR}, we have
    \[\Pro(G)=\Pro(G/\bZ(G))=\Pro(\Sigma_{3})=\frac{1}{2}.\] The other direction of the claim follows from the above-mentioned theorem of Lescot \cite{Lescot} since if $G$ is isoclinic to $\Sigma_3$, then $G/\bZ(G)=\Sigma_{3}$.
    \end{proof}


\section{Hall $\pi$-subgroups}\label{sec:Hall}

In this section we prove that the second statement of Theorem \ref{thm:main} follows from the first.

Let $\mathcal{D}_\pi$ be the collection of all finite groups $G$
such that $G$ has a Hall $\pi$-subgroup, any two Hall
$\pi$-subgroups of $G$ are conjugate, and any $\pi$-subgroup of $G$
is contained in a Hall $\pi$-subgroup. Of course $\mathcal{D}_\pi$
is everything when $\pi$ is a single prime by Sylow's theorems.
Also, $\mathcal{D}_\pi$ contains all $\pi$-separable groups. The
following easy observation is useful to bound $\dd_{\pi}(G)$ in the
case $G\in \mathcal{D}_{\pi}$.

\begin{lem}\label{lem3}
Let $G$ be a finite group in $\mathcal{D}_{\pi}$. If $H$ is a Hall
$\pi$-subgroup of $G$, then $$\dd_{\pi}(G) \leq \Pro(H).$$
\end{lem}

\begin{proof}
Since $|H|=|G|_{\pi}$, it suffices to see that $k_{\pi}(G) \leq
k(H)$. If $x, y \in H$ are not conjugate in $G$, then they cannot be
conjugate in $H$. Since $G \in \mathcal{D}_{\pi}$, every $G$-class of $\pi$-elements has a
representative in $H$.
\end{proof}

From this, we can easily prove Theorem \ref{thm:main} in  case $G
\in \mathcal{D}_{\pi}$.

\begin{thm}\label{teoC1}
Let $\pi$ be a set of primes and $G$ a finite group in
$\mathcal{D}_{\pi}$. Then Theorem \ref{thm:main} holds for $G$.
\end{thm}

\begin{proof}
By hypothesis, $G$ has a Hall $\pi$-subgroup $H$ and all the Hall
$\pi$-subgroups of $G$ are $G$-conjugates of $H$. Thus, by Lemma
\ref{lem3}, we have $\dd_{\pi}(G) \leq \Pro(H)$. Let $p$ be the
smallest prime in $\pi$. Assume that $\dd_{\pi}(G)> \frac{1}{p}$. We
then have
$$\Pro(H)>\frac{1}{p}.$$
Theorem \ref{firstpart} and Lemma \ref{Gprima} then imply that
$|H'|\leq p$ and $H$ is nilpotent, as claimed. Moreover, if
$\dd_{\pi}(G)> \frac{p^2+p-1}{p^3}$ then $H$ is abelian by Lemma \ref{lem1}.
\end{proof}

As a consequence of Theorem \ref{teoC1}, we have that Theorem
 \ref{thm:main} holds if $\pi=\{p\}$ or if $G$ is $\pi$-separable.

We also recall some facts on the groups in $\mathcal{D}_\pi$.
The first one is a result of Wielandt \cite{Wielandt} mentioned in
the Introduction and the second one is due to Hall \cite[Theorem
D5]{Hall}.

\begin{lem}\label{Wielandt}
Let $G$ be a finite group and  $\pi$ a set of primes.

\begin{enumerate}

\item If $G$ possesses a nilpotent Hall $\pi$-subgroup, then  $G \in
\mathcal{D}_{\pi}$.

\item If $N$ possesses nilpotent Hall $\pi$-subgroups, $G/N$ possesses
solvable Hall $\pi$-subgroups, and $G/N \in \mathcal{D}_{\pi}$, then
$G \in \mathcal{D}_{\pi}$.
\end{enumerate}
\end{lem}

\begin{thm}\label{equivalency}
    \label{3.4}
The second statement of Theorem \ref{thm:main} follows from the first.
\end{thm}

\begin{proof}
 Let $G$ be a group with
 $\dd_{\pi}(G)>\frac{p^2+p-1}{p^3}>\frac{1}{p}$.
 By hypothesis,
$G$ possesses  a nilpotent Hall $\pi$-subgroup. It then follows that
$G \in \mathcal{D}_{\pi}$ by Lemma \ref{Wielandt}. The result follows by Theorem
\ref{teoC1}.
\end{proof}

The rest of the paper is therefore devoted to prove that $G$ has a
nilpotent Hall $\pi$-subgroup under the condition
$\dd_{\pi}(G)>\frac{1}{p}$.


\section{Reducing to a problem on simple groups}\label{sec:reducing}

In this section we prove Theorem \ref{thm:main}, assuming a result on finite simple
groups.

\subsection{Reducing to simple groups}

We begin by recalling two properties of $\dd_{\pi}(G)$. The first
one is \cite[Proposition 5]{MN}, essentially due to Robinson.
The second is due to Fulman and Guralnick \cite[Lemma 2.3]{FG}.

\begin{lem}\label{dpi}
Let $G$ be a finite group and $\pi$ a set of primes.

\begin{enumerate}[\rm(i)]
\item Let $\mu \subseteq \pi$. Then
$\dd_{\pi}(G)\leq \dd_{\mu}(G)$.

\item
$\dd_{\pi}(G) \leq \dd_{\pi}(N)\dd_{\pi}(G/N)$ for any normal
subgroup $N$ of $G$.
\end{enumerate}
\end{lem}

\begin{lem}\label{Sylowab}
Let $G$ be a finite group, $\pi$  a set of primes, and $p$ the
smallest prime in $\pi$. Let $q\in\pi$ and $Q \in \Syl_{q}(G)$.
Suppose $\dd_{\pi}(G)>\frac{1}{p}$. We have

\begin{enumerate}[\rm(i)]
\item $Q/\bZ(Q)$ is abelian and $|Q'|\leq q$.
\item If $q \in \pi \setminus \{p\}$, then $Q$ is abelian.
\end{enumerate}
\end{lem}

\begin{proof}
By Sylow's theorems and Lemma \ref{lem3} we have $\dd_{q}(G)\leq
\Pro(Q)$. On the other hand, by Lemma \ref{dpi}(i), we have
$\dd_{\pi}(G)\leq \dd_{q}(G)$. We deduce that
$$\frac{1}{q}\leq \frac{1}{p}<\Pro(Q).$$
Theorem \ref{firstpart} and Lemma \ref{Gprima} now imply that
$Q/\bZ(Q)$ is abelian and $|Q'|\leq q$.

Suppose $q > p$. Then $q\geq p+1$, and one can easily
check that $\frac{q^2+q-1}{q^3}<\frac{1}{p}$. Now
$\Pro(Q)>\frac{q^2+q-1}{q^3}$, and thus $Q$ must be abelian
by Lemma \ref{lem1}.
\end{proof}

The next lemma is  \cite[Lemma 3.1]{Alex}, which allows us to work
with a set of two primes instead of an arbitrary set.

\begin{lem}[Moret\'{o}]\label{Alex}
Let $G$ be a finite group and let $\pi$  a set of primes. If $G$
possesses a nilpotent Hall $\tau$-subgroup for every $\tau\subseteq
\pi$ with $|\tau|=2$, then $G$ possesses a nilpotent Hall
$\pi$-subgroup.
\end{lem}

\begin{pro}\label{reduc}
Suppose that Theorem \ref{thm:main} is false for  a group $G$. Then
there exists $\pi=\{p,q\}$, where $p<q$ are two primes, such that
$G$ does not possess nilpotent Hall $\pi$-subgroups and for all $P
\in \Syl_{p}(G)$ and $Q\in \Syl_{q}(G)$, $P/\bZ(P)$ is abelian,
$|P'|\leq p$, and $Q$ is abelian.
\end{pro}

\begin{proof}
By Theorem \ref{equivalency}, we may assume that there exists $\pi$,
 a set of primes, such that $\dd_{\pi}(G)>\frac{1}{p}$, but $G$ does
not possess nilpotent Hall $\pi$-subgroups, where $p$ is the
smallest member of $\pi$.

If $G$ has a nilpotent Hall $\tau$-subgroup  for every $\tau
\subseteq \pi$ with $|\tau|=2$, then by Lemma \ref{Alex}, $G$ has
nilpotent Hall $\pi$-subgroups. Thus, there exists $\{q,r\}\subseteq
\pi$ with $q<r$ such that $G$ does not possess a nilpotent Hall
$\{q,r\}$-subgroup. By Lemma \ref{dpi}(i), we also have
$\dd_{\pi}(G)\leq \dd_{\{q,r\}}(G)$, and it follows that
$$ \frac{1}{q} \leq \frac{1}{p} < \dd_{\pi}(G) \leq \dd_{\{q,r\}}(G).$$
Therefore, Theorem \ref{thm:main} fails for $G$ and the set
$\{q,r\}$, and hence we may assume that $|\pi|=2$, that is $\pi=\{p,q\}$ with $p<q$.

Finally, the assertion on the Sylow subgroups follows from Lemma
\ref{Sylowab}.
\end{proof}

\begin{pro}\label{simple}
Let $\pi$ be a set of primes and $p$ the smallest member in $\pi$.
Let $G$ be a finite group with minimal order subject to the
conditions that $\dd_{\pi}(G)>\frac{1}{p}$ and $G$ does not possess
nilpotent Hall $\pi$-subgroups. Then $G$ is non-abelian simple.
\end{pro}

\begin{proof}
We may assume that $G$ is non-abelian and not simple. Let $N$ be a
nontrivial proper normal subgroup in $G$. By Lemma \ref{dpi}(ii), we
have
$$\frac{1}{p}<\dd_{\pi}(G)\leq \dd_{\pi}(G/N)\dd_{\pi}(N).$$
It follows that $\frac{1}{p}<\dd_{\pi}(G/N)$ and
$\frac{1}{p}<\dd_{\pi}(N)$, as both $d_\pi(N)$ and $d_\pi(G/N)$ are
at most one (see \cite[Lemma 3.5]{GGG}). By the minimality of $G$,
$N$ and $G/N$ possess nilpotent Hall $\pi$-subgroups. Applying Lemma
\ref{Wielandt}, we then deduce that both $N$ and $G/N$ are members
of $\mathcal{D}_{\pi}$. It follows that $G/N \in \mathcal{D}_{\pi}$,
$G/N$ possesses solvable Hall $\pi$-subgroups and $N$ possesses
nilpotent Hall $\pi$-subgroups. By Lemma \ref{Wielandt}(ii), we have
$G \in \mathcal{D}_{\pi}$. Therefore, by Theorem \ref{teoC1}, we
have that $G$ possesses nilpotent Hall $\pi$-subgroups, which is a
contradiction. We conclude that $G$ is non-abelian simple.
\end{proof}

\subsection{Reducing to a question on simple groups}

The following is a consequence of a result of Tong-Viet, which
asserts that if $\dd_2(G)>\frac{1}{2}$ then $G$ possesses a normal
$2$-complement.

\begin{lem}\label{case2}
Let $S$ be a non-abelian simple group and $\pi$ be a set of primes
containing $2$. Then $\dd_{\pi}(S)\leq \frac{1}{2}$.
\end{lem}

\begin{proof}
Suppose that $\dd_{\pi}(S)>\frac{1}{2}$. Then $\frac{1}{2}<
\dd_{\pi}(S)\leq \dd_2(S)$. By \cite[Theorem A]{TV}, $S$ possesses a
normal 2-complement, which is impossible.
\end{proof}

\begin{pro}\label{precon1}
Let $G$ be a group and $\pi$ a set of primes such that
$\dd_{\pi}(G)>\frac{1}{p}$, where $p$ is the smallest prime in
$\pi$. Let $q\in \pi$ but $q\neq p$. Then  $q$ does not divide
$|\mathbf{N}_{G}(P):\mathbf{C}_{G}(P)|$ where $P \in
\Syl_{p}(G)$.
\end{pro}

\begin{proof}
Assume by contradiction that $q$  divides
$|\mathbf{N}_{G}(P)/\mathbf{C}_{G}(P)|$. Let $x$ be an element of
order $q$ in $\mathbf{N}_{G}(P)/\mathbf{C}_{G}(P)$ where $P \in \mathrm{Syl}_{p}(G)$. Consider the
action of $X=\langle x\rangle$ on $P$. Let $r$ be the number of
elements of $P$ fixed by $X$.

We claim that $r>\frac{|P|}{p^2}$. Assume to the contrary that
$r\leq \frac{|P|}{p^2}$. We have $|P|= r+t\cdot q$, implying
that $t=\frac{|P|-r}{q}$. Since each $X$-orbit on $P$ is contained
in a conjugacy class  of $p$-elements it is easy to see that
$k_{p}(G)\leq r+t$.  Now we have
$$\frac{1}{p}<\dd_{\pi}(G)\leq \dd_{p}(G)=\frac{k_{p}(G)}{|P|}\leq \frac{r+t}{|P|}=
\frac{1}{q}\left((q-1)\frac{r}{|P|}+1\right)\leq
\frac{1}{q}\left((q-1)\frac{1}{p^2}+1\right).$$ It is not hard to
see that this implies $q\leq p$, which is a contradiction. We have
shown that $r>\frac{|P|}{p^2}$.

Since $r$ divides $|P|$, it follows that $$r\in \{|P|,{|P|}/{p}\}.$$
If $r=|P|$ then $X$ centralises $P$, which is impossible. Thus
$r={|P|}/{p}$ and hence there exists  a subgroup $H$ of order
${|P|}/{p}$ that is centralised by $X$. That is,
$$H=\mathbf{C}_{P}(X)=\{z \in P\mid z^x=z \text{ for all } x \in X\}.$$

Let $L:=P:X$ be the semidirect product of the relevant action of $X$
on $P$. Then $L/H\cong C:X$ for some $C\cong C_p$. Since $H$ is
maximal in $P$, the subgroup $H$ is normal in $P$, and it is
$X$-invariant, applying \cite[Corollary 3.28]{Isaacs}, we have
$$\mathbf{C}_{P/H}(X)=\mathbf{C}_{P}(X)H/H=H/H,$$
and hence $X$ acts nontrivially on $C$. Let $\oo$ be a nontrivial
orbit of the action of $X$ on $C$. We now have $q=|X|=|\oo| \leq |C|=p$, which is
a contradiction.
\end{proof}

\begin{cor}\label{precon2}
Let $G$ be a group and  $\pi=\{p,q\}$ a set of primes with $p<q$
such that $\dd_{\pi}(G)>\frac{1}{p}$. Let $P \in \Syl_{p}(G)$. Then
$q$ divides $|\Syl_{p}(G)|=|G:\mathbf{N}_{G}(P)|$ or $G$ possesses a
nilpotent Hall $\pi$-subgroup.
\end{cor}

\begin{proof}
We know that $|G|_{q}$ divides
$$|G|=|G:\mathbf{N}_{G}(P)||\mathbf{N}_{G}(P):\mathbf{C}_{G}(P)||\mathbf{C}_{G}(P)|$$ but $q$ cannot divide
$|\mathbf{N}_{G}(P):\mathbf{C}_{G}(P)|$ by Proposition
\ref{precon1}. Assume that $q$ does not divide
$|G:\mathbf{N}_{G}(P)|$. Then $|G|_{q}$ divides
$|\mathbf{C}_{G}(P)|$. Therefore, there exists $Q \in \Syl_{q}(G)$
with $Q \leq \mathbf{C}_{G}(P)$. Now $PQ$ is a nilpotent Hall
$\pi$-subgroup of $G$.
\end{proof}

Now we can prove Theorem \ref{thm:main}, modulo the following
statement about simple groups whose proof is deferred to the next
section.

\begin{thm}\label{teoB}
Let $G$ be a non-abelian simple group and  $\pi=\{p,q\}$ be a set of
two odd primes with $p<q$. Assume that there exist $P \in
\Syl_{p}(G)$ and $Q\in \Syl_{q}(G)$ such that $P/\bZ(P)$ is abelian,
$|P'|\leq p$, $Q$ is abelian, and $q$ divides
$|G:\mathbf{N}_{G}(P)|$. Then
$\dd_{\pi}(G)\leq \frac{1}{p}$.
\end{thm}

Observe that in Theorem \ref{teoB} we may assume that both $p$ and $q$ divide the order of $G$.

\begin{thm}
    \label{4.10}
Let $G$ be a finite group, $\pi$ be a set of primes, and $p$ be the
smallest prime in $\pi$. Assume Theorem \ref{teoB}. If
$\dd_{\pi}(G)> \frac{1}{p}$ then $G$ has a nilpotent Hall
$\pi$-subgroup.
\end{thm}

\begin{proof}
Assume that the theorem is false and let $G$ be a minimal
counterexample. In particular, $\dd_{\pi}(G)> \frac{1}{p}$ but $G$
has no nilpotent Hall $\pi$-subgroups. By Proposition \ref{simple},
we know that $G$ is non-abelian simple. Using Lemma \ref{case2}, we know furthermore that $p\not=2$.

By Proposition \ref{reduc},
there exists $\pi=\{p,q\}$ with (odd) $p<q$ such that $\dd_{\pi}(G)>
\frac{1}{p}$, $P/\bZ(P)$ is abelian, $|P'|\leq p$, and $Q$ is
abelian, where $P\in\Syl_p(G)$ and $Q\in\Syl_q(G)$. We also have that
$q$ divides $|G:\mathbf{N}_{G}(P)|$, by Corollary \ref{precon2}. We
now have all the hypotheses of Theorem \ref{teoB}, and therefore
deduce that $\dd_{\pi}(G)\leq \frac{1}{p}$. This is a contradiction.
\end{proof}

We remark that we have indeed proved Theorem \ref{thm:main} when the
set $\pi$ contains the prime $2$, and this result does not rely on
the classification of finite simple groups.


\section{Simple groups}\label{sec:simple groups}

In this section we prove Theorem \ref{teoB}, by using the
classification. We begin with the
alternating groups.

\begin{lem}\label{sylalt}
Let $p$ be an odd prime, $n\geq 5$ be an integer and $P \in
\Syl_{p}(\Al_{n})$.
\begin{enumerate}[\rm(i)]
\item If $n\geq p^2$, then $P/\bZ(P)$ is not abelian.

\item If $n<p^2$, then $P$ is elementary abelian.
\end{enumerate}
\end{lem}

\begin{proof}
For (i) it is sufficient to exhibit a subgroup $H$ of $P$ such that
$H/\bZ(H)$ is not abelian. If $n \geq p^{2}$, then $H = C_{p} \wr
C_{p}$ is such a subgroup of $P$. Statement (ii) follows from the
description of the Sylow $p$-subgroups of $\Al_{n}$ found in
\cite[Satz III.15.3]{Huppert}.
\end{proof}

\begin{thm}\label{prop:An}
Let $n\geq 5$, $\pi=\{p,q\}$ be a set of two odd  primes with $p<q$,
and $P\in \Syl_{p}(\Al_{n})$. Assume that both $p$ and $q$ divide the order of $\Al_{n}$. If $P/\bZ(P)$ is abelian, then
$\dd_{\pi}(\Al_{n})\leq \frac{1}{p}$. In particular, Theorem
\ref{teoB} holds for alternating groups.
\end{thm}

\begin{proof}
Let $P \in \Syl_{p}(\Al_{n})$ and $Q \in \Syl_{q}(\Al_{n})$. Since
$P/\bZ(P)$ is abelian, $n < p^{2}$ by Lemma \ref{sylalt}. Let $n=rp+s=lq+t$, where $r,s\in \{0,1,\ldots,p-1\}$ and
$l,t \in \{0,1,\ldots,q-1\}$. Then $P=(C_{p})^{r}$ and
$Q=(C_{q})^{l}$ with both $r$ and $l$ at least $1$.

It is easy to see that every $\pi$-element of $\Al_{n}$ can be
expressed as a product of the form $xy=yx$, where $x$ is a product
of cycles of length $p$ and $y$ is a product of cycles of length
$q$. Since $n < p^{2}$, the supports of $x$ and $y$ are disjoint.

Assume first that $n \geq p+q+2$. In this case we have that
$k_{p}(\Al_{n})=1+r\leq p$, $k_{q}(\Al_{n})=1+l\leq q$ and
$|\Al_{n}|_{\pi}=p^{r}q^{l}$. Thus we have
$$\dd_{\pi}(\Al_{n})=\frac{k_{\pi}(\Al_{n})}{|\Al_{n}|_{\pi}}\leq \frac{pq}{p^{r}q^{l}}.$$
If $(r,l)\not=(1,1)$, then $\dd_{\pi}(\Al_{n})\leq
\frac{1}{p}$. Assume now that $r=l=1$. Then $k_{\pi}(\Al_{n})\leq
k_{p}(\Al_{n})k_{q}(\Al_{n})=4$ and hence $\dd_{\pi}(\Al_{n})\leq
\frac{4}{q}\frac{1}{p}<\frac{1}{p}$, where the last inequality holds
because $q \geq 5$.

Assume now that $n\leq p+q+1$ and so $l = 1$. In this case it may happen that a
$\Sigma_{n}$-conjugacy class of $\pi$-elements splits in two
different $\Al_{n}$-conjugacy classes. We thus have $k_{\pi}(\Al_{n})\leq
(1+r)(1+l) + 1 = 2(1+r)+1 = 2r+3$. It follows that $\dd_{\pi}(\Al_{n}) \leq \frac{2r+1}{p^{r}q}$. If $r \geq 2$, then $\frac{2r+3}{p^{r}q} < \frac{1}{q} < \frac{1}{p}$. If $r=1$, then $2r+3 = 5 \leq q$ and so once again $\dd_{\pi}(\Al_{n}) \leq \frac{1}{p}$.
\end{proof}

For convenience, we will consider the Tits group ${}^2F_4(2)'$ as a
sporadic simple group.

\begin{thm}\label{prop:sporadic}
Let $S$ be a sporadic simple group and $\pi = \{ p, q \}$ where $p <
q$ are odd  primes dividing $|S|$. If $(S,\pi)\not=(J_1,\{3,5\})$
then $\dd_{\pi}(S) \leq \frac{1}{p}$. In particular, Theorem
\ref{teoB} holds for $S$.
\end{thm}

\begin{proof}
In what follows we use information in \cite{Atlas} without further
notice. We may assume that $\pi$ is a set of primes such that
$k_{\pi}(S) \geq 6$, for otherwise
$$\dd_{\pi}(S) = \frac{k_{\pi}(S)}{|S|_{\pi}} \leq \frac{5}{pq} \leq  \frac{1}{p}$$
There is no such $\pi$ for the four smallest Mathieu groups. For
each of the groups $M_{24}$, $HS$, $J_{2}$ there are two
possibilities for $\pi$. In each of the six cases $k_{\pi}(S)$ is at
most $|S|_{p}$ or $|S|_{q}$ and this is sufficient to obtain the
bound $\dd_{\pi}(S) \leq  \frac{1}{p}$.

So we assume that $S$ is not one of the groups already analyzed. If
$S$ is different from $Fi_{23}$, $Fi_{24}'$ and $J_1$, then we count
the total number of conjugacy classes of $S$ of elements of odd
order. These numbers are usually less that $|S|_r$ for a given prime
divisor $r$ of $|S|$. If this is the case for a prime $r$, then we
can assume that $r$ does not lie in $\pi$ (otherwise we would be
done). This gives strong restrictions on the set $\pi$. In fact,
given that $k_{\pi}(S) \geq 6$, we find this way that $S$ must be
$J_4$ and $\pi$ is either $\{ 3,7 \}$ or $\{ 5,7 \}$. In each of
these two cases we count the number of $\pi$-classes in $S$ to
obtain our bound of $ \frac{1}{p}$ for $\dd_{\pi}(S)$.

If $S$ is $Fi_{23}$ or $Fi_{24}'$, then we again count the number of
conjugacy classes of $S$ of elements of odd order. This allows us to conclude
that $3$ cannot lie in $\pi$.
We then count the number of conjugacy classes of $S$ whose elements have orders
divisible neither by $2$ nor $3$. This number is $8$ in the first case
and $14$ in the second. By
looking at the prime factorization of $|S|$, the only case to consider
is $S = Fi_{24}'$ and $\pi = \{ 11, 13 \}$. But it turns out that
$k_{\pi}(S) = 3$ in this case.

The only group remaining is $S = J_1$. The number of conjugacy classes
of $S$ of elements of odd order is $11$ forcing $\pi$ to be a subset of $\{ 3, 5, 7
\}$. Then $k_{\pi}(S) = 3$
or $\pi = \{ 3, 5 \}$ and $k_{\pi}(S) = 6$, giving $\dd_{\pi}(S) = \frac{2}{5}$.

The last assertion follows from the fact that if $P \in
\Syl_3(J_{1})$, then $5$ does not divide
$|J_1:\mathbf{N}_{J_{1}}(P)|$.
\end{proof}

We are left with the case of simple groups of Lie type $S\neq
{}^2F_4(2)'$. For the sake of convenience, we rename the prime $q$
in Theorem \ref{teoB} to $s$ in order to reserve $q$ for the size of
the underlying field of $S$.

The proof of Theorem \ref{teoB} for groups of Lie type is divided
into two fundamentally different cases: $\pi$ contains the defining
characteristic of $S$ and $\pi$ does not. The former case is fairly
straightforward.

\begin{thm}\label{thm:defining char}
Let $S$ be a finite simple group of Lie type in characteristic $p>2$
and $\pi=\{p,s\}$, where $s$ is an odd prime dividing $|S|$. Then,
$$\dd_{\pi}(S)\leq \frac{1}{s}.$$
In particular, Theorem \ref{teoB} holds for simple groups of Lie
type when $\pi$ contains the defining characteristic of $S$.
\end{thm}

\begin{proof}
First we observe that the desired inequality is satisfied if
$k_\pi(S)\leq |S|_p$. We shall make use of well-known bounds of Fulman and Guralnick \cite{FG} for the numbers of conjugacy classes of finite Chevalley
groups to show that, when $S$ has high enough rank, even the
stronger inequality $k(S)\leq |S|_p$ holds true.

Let $S=\PSL(n,q)$. Then $k(S)\leq \min\{2.5
q^{n-1},q^{n-1}+3q^{n-2}\}$ by \cite[Proposition 3.6]{FG}. This is
certainly smaller than $|S|_p=q^{n(n-1)/2}$ if $n\geq 4$. Therefore,
we just need to verify the theorem for $n=2$ or $3$. The theorem is
in fact straightforward to verify for  these low rank cases, using
the known information on conjugacy classes of the group (see
\cite[Chapter 38]{Dornhoff} for $n=2$ and \cite{SF73} for $n=3$).
The case $S=\PSU(n,q)$ is entirely similar.

Next, we consider $\PSp(2n,q)$ with $n\geq 3$. Then $k(S)\leq 10.8
q^n$ for odd $q$, and it easily follows that $k(S)\leq
|S|_p=q^{n^2}$. The case of orthogonal groups is similar, with a
remark that $k(\Omega(2n+1,q))\leq 7.3 q^n$ for $n\geq 2$ and
$k(\mathrm{P}\Omega^\pm(2n,q))\leq 6.8 q^n$ for $n\geq 4$.

Now we turn to exceptional groups. Recall that the defining
characteristic $p$ of $S$ is odd, so we will exclude the types
${}^{2}_{}B_{2}$ and ${}^{2}_{}F_{4}$. By \cite[Table 1]{FG} (or
\cite{Lubeck} for more details), we observe that $k(S)$ is bounded
above by a polynomial with positive coefficients, say $g_{S}$,
evaluated at $q$. Suppose $S$ is one of ${}^{3}_{}D_{4}(q)$,
$F_{4}(q)$, $E_{6}(q)$, ${}^{2}_{}E_{6}(q)$, $E_{7}(q)$, or
$E_{8}(q)$. We then have
$$k(S)\leq g_S(1)q^{\deg(g_S)}\leq 252q^{\deg(g_S)} \text{  and }
\frac{q^{\deg(g_S)}}{|S|_{p}}\leq \frac{1}{q^8}.$$ Therefore,
$$\dd_{\pi}(S)\leq \frac{k(S)}{|S|_{p}|S|_{s}}\leq
\frac{252}{sq^8}<\frac{1}{s},$$ as wanted. The remaining cases of
the types $G_2$ and ${}^2G_2$ are even easier, using the more
refined bounds $k(G_{2}(q))\leq q^2+2q+9$ and
$k({}^{2}_{}G_{2}(q))\leq q+8$.
\end{proof}

\begin{lem}\label{centers}
Let $G$ be a finite group and let $\pi$ be a set of primes such that
$|\bZ(G)|_{\pi}=1$. Then, $k_{\pi}(G)=k_{\pi}(G/\bZ(G))$.
\end{lem}

\begin{proof}
Let $Z: = \bZ(G)$. Every coset $gZ$ of $Z$ in $G$ contains at most
one $\pi$-element  of $G$ since $|Z|_{\pi} = 1$. The $\pi$-elements
of $G/Z$ are $gZ$ where $g$ runs  through the $\pi$-elements of $G$.
If $g$ is a $\pi$-element, then the conjugacy  class of $gZ$ in
$G/Z$ consists of $hZ$ where $h\in g^G$. Thus, there is a bijection
between the $\pi$-conjugacy classes of $G$ and the $\pi$-conjugacy
classes of $G/Z$.
\end{proof}

In the case when $\pi$ does not contain the defining characteristic
of $S$, the conjugacy classes of $\pi$-elements of $S$ will be
semisimple classes, which can be conveniently described via an
ambient algebraic group of $S$ and its Weyl group.

It is well-known that every simple group of Lie type $S\neq
{}^2F_4(2)'$ is of the form
$S=\textbf{G}^{F}/\mathbf{Z}(\textbf{G}^{F})$ for some simple
algebraic group $\textbf{G}$ of simply connected type and a
suitable Steinberg endomorphism $F$ on $\mathbf{G}$, see \cite[Theorem
24.17]{MT} for instance.

\begin{thm}\label{quasisimple}
Let $S$ be a finite simple group of Lie type and $\mathbf{G}, F$ as
above. Let $\pi=\{p,s\}$ with $p<s$ be a set of primes not
containing the defining characteristic of $S$. Suppose that $s$
divides $|\Syl_{p}(S)|$. Then \[\dd_{\pi}(\textbf{G}^{F})\leq
\frac{1}{p}.\] In particular, if
$|\mathbf{Z}(\textbf{G}^{F})|_{\pi}=1$, then $\dd_{\pi}(S)\leq
\frac{1}{p}$.
\end{thm}

\begin{proof}
Let $G:=\textbf{G}^F$.  We first claim that a Hall $\pi$-subgroup of
$G$, if exists, cannot be abelian. Assume by contradiction that $G$
does have such subgroup, say $H$. Then
$\overline{H}:=H\bZ(G)/\bZ(G)$ would be an abelian Hall
$\pi$-subgroup of $S$, implying that $\mathbf{N}_S(P)$ contains
$\overline{H}$, where $P$ is a Sylow $p$-subgroup of $S$ that is
contained in $\overline{H}$. It follows that $s$ does not divide
$|S:\mathbf{N}_S(P)|$, violating the hypothesis.

Let $\textbf{T}$ be an $F$-stable  maximal  torus of $\mathbf{G}$,
and let $W=N_{\textbf{G}}(\textbf{T})/\textbf{T}$ be the Weyl group
of $\textbf{G}$. Since $\pi$ does not contain the defining
characteristic of $S$, the conjugacy classes of $\pi$-elements of $G$
are semisimple classes. According to \cite[Proposition
3.7.3]{Carter85} and its proof, there is a well-defined bijection
\[
\tau:\mathrm{Cl}_{ss}(G)\rightarrow (\mathbf{T}/W)^F
\]
between the set $\mathrm{Cl}_{ss}(G)$ of semisimple conjugacy
classes of $G$ and the set $(\mathbf{T}/W)^F$ of $F$-stable orbits
of $W$ on $\mathbf{T}$. Malle, Navarro, and Robinson showed in
\cite[Theorem 3.15]{GGG} that this bijection $\tau$ preserves
element orders, and therefore the counting formula (and its proof)
for the number of $F$-stable orbits of $W$ on $\mathbf{T}$ in
\cite[Proposition 3.7.4]{Carter85} implies that
$$k_{\pi}(G)=\frac{1}{|W|}\sum_{w \in W}|\textbf{T}^{w^{-1}F}|_{\pi}.$$
It follows that
$$\dd_{\pi}(G)=\frac{1}{|W|}\sum_{w \in W}\frac{|\textbf{T}^{w^{-1}F}|_{\pi}}{|G|_{\pi}}.$$

Now, if $|\textbf{T}^{w^{-1} F}|_\pi=|G|_\pi$ for some $w \in W$ then a Hall
$\pi$-subgroup of $\textbf{T}^{w^{-1} F}$, which is abelian, would
be a Hall $\pi$-subgroup of $G$, and this contradicts the above
claim. Thus \[\frac{|\textbf{T}^{w^{-1} F}|_\pi}{|G|_\pi}\leq 1/p\]
for every $w\in W$. It then follows that
$$\dd_{\pi}(G)\leq \frac{1}{p},$$ proving the first part of the
theorem.

For the second part, assume that $|\bZ(G)|_{\pi}=1$. By Lemma
\ref{centers}, we then have
$$\dd_{\pi}(S)=\dd_{\pi}(G/\bZ(G))=\dd_{\pi}(G)\leq \frac{1}{p},$$
as stated.
\end{proof}

Theorem \ref{quasisimple} already proves Theorem \ref{teoB} in
several cases, as seen in the next result. In what follows, to unify
the notation, we use $\GL^\epsilon$, $\SL^\epsilon$ and
$\PSL^{\epsilon}$ for linear groups when $\epsilon=+$ and for
unitary groups when $\epsilon=-$. We also use $E_6^+$ for $E_6$ and
$E^-_6$ for ${}^2E_6$.

\begin{thm}\label{prop:ABC}
Let $S$ be a simple group of Lie type, $\pi$ be a set of two odd
primes not containing the defining characteristic of $S$, and $p$ be
the smaller prime in $\pi$. Assume that we are not in one of the
following situations:

\begin{itemize}
\item[(i)] $S=E^\epsilon_{6}(q)$ and $3 \in \pi$.

\item [(ii)] $S=\PSL^{\epsilon}(n,q)$ with $n \geq 3$ and
$\gcd(n,q-\epsilon)_{\pi}\not=1$.
\end{itemize}
Then $\dd_{\pi}(S)\leq \frac{1}{p}$.
\end{thm}

\begin{proof}
Let $\mathbf{G}$ and $F$ be as in Theorem \ref{quasisimple}.
According to \cite[Table 24.12]{MT}, if we are not in one of the
stated situations, then $|\bZ(\mathbf{G}^F)|_\pi=1$. The result then
follows from Theorem \ref{quasisimple}.
\end{proof}

Next we prove Theorem \ref{teoB} for case (i) in Theorem
\ref{prop:ABC}.

\begin{pro}\label{prop:E6}
Let $S= E^\epsilon_{6}(q)$ with $(3,q)=1$ and $P \in \Syl_{3}(S)$.
Then $|P'|>3$. In particular, Theorem \ref{teoB} holds in the case
$S= E^\epsilon_{6}(q)$ and $3\in\pi$.
\end{pro}

\begin{proof}
Let $\textbf{G}$ be a simple algebraic group of simply connected
type and $F:\textbf{G} \rightarrow \textbf{G}$ a Frobenius map such
that $S=\textbf{G}^{F}/\bZ(\textbf{G}^{F})$. By \cite[Theorem
25.17]{MT}, we know that every Sylow $3$-subgroup of $\mathbf{G}^F$
lies in $\mathbf{N}_{\mathbf{G}^F}(\textbf{T})$ for some maximal
$F$-stable torus $\textbf{T}$ of $\mathbf{G}$. Therefore Sylow
$3$-subgroups of
$\mathbf{N}_{\mathbf{G}^F}(\textbf{T})/\textbf{T}^F=\mathrm{SO}(5,3)$
(the Weyl group of $E_6$) are homomorphic images of Sylow
$3$-subgroups of $S=\mathbf{G}^F/\bZ(\mathbf{G}^F)$. Since the size
of the derived subgroup of Sylow $3$-subgroups of $\mathrm{SO}(5,3)$
is $9$, we deduce that $|P'|>3$.
\end{proof}

For the rest of this section, we will prove Theorem \ref{teoB} for
case (ii) in Theorem \ref{prop:ABC}.

\begin{lem}\label{lem:PSL-PSU}
Let $p$ be an odd prime and $S=\PSL^\epsilon(n,q)$. Assume that $p$
divides  $\gcd(n,q-\epsilon)$ and Sylow $p$-subgroups of $S$ are
abelian. Then $n=p=3$. Furthermore, $q-\epsilon$ is divisible by $3$
but not $9$.
\end{lem}

\begin{proof}
It is argued in Lemma 2.8 of \cite{KS21} that if Sylow $p$-subgroups
of $S$ are abelian and $p\geq 5$ then $p$ cannot divide
$|\bZ(\SL^\epsilon(n,q))|$. Therefore our hypotheses imply that
$p=3$.

We first prove that $n=3$. The condition $p=3$ divides
$\gcd(n,q-\epsilon)$, implies that $n\geq 3$. Assume by
contradiction that $n>3$. Let $w$ be the (unique) element of order $3$
of $\mathbb{F}^{\times}_{q^2}$, and consider the element
$g:=\diag(I_{n-2},w,w^{-1})$. We have
\[
\mathbf{C}_{\GL^\epsilon(n,q)}(g)=\GL^\epsilon(n-2,q)\times
\GL^\epsilon(1,q)^2,
\]
and so
\[
|\GL^\epsilon(n,q):\mathbf{C}_{\GL^\epsilon(n,q)}(g)|=
q^{2n-1}\frac{(q^n-\epsilon^n)(q^{n-1}-\epsilon^{n-1})}{(q-\epsilon)^2}.
\]
Since $3$ divides $\gcd(n,q-\epsilon)$, we have that  $3$ must
divide $ |\GL^\epsilon(n,q):\mathbf{C}_{\GL^\epsilon(n,q)}(g)|$. In
fact, we also have $3$ divides
$|\SL^\epsilon(n,q):\mathbf{C}_{\SL^\epsilon(n,q)}(g)|$. On the
other hand, as $1$ is the only eigenvalue of $g$ with multiplicity
larger than 1 (recall that $n>3$), it is easy to see that
$\mathbf{C}_{\SL^\epsilon(n,q)}(g)$ is the full pre-image of
$\mathbf{C}_{\PSL^\epsilon(n,q)}(\overline{g})$ under the natural
projection from $\SL^\epsilon$ to $\PSL^\epsilon$, where
$\overline{g}$ is the image of $g$ in $\PSL^\epsilon(n,q)$. In
particular,
$|\SL^\epsilon(n,q):\mathbf{C}_{\SL^\epsilon(n,q)}(g)|=|\PSL^\epsilon(n,q):\mathbf{C}_{\PSL^\epsilon(n,q)}(\overline{g})|$,
and hence $3 $ divides
$|\PSL^\epsilon(n,q):\mathbf{C}_{\PSL^\epsilon(n,q)}(\overline{g})|$,
implying that Sylow $3$-subgroups of $S=\PSL^\epsilon(n,q)$ are not
abelian. We have shown that $n=3$.

Finally, assume that $9$ divides $q-\epsilon$. Let $\lambda$ be the
element of order 9 in $\mathbb{F}^{\times}_{q^2}$ and consider
$h:=\diag(\lambda,\lambda^3,\lambda^5)\in \SL^\epsilon(3,q)$, also
of order 9. We then have
$\mathbf{C}_{\GL^\epsilon(3,q)}(h)=\GL^\epsilon(1,q)^3$, so that
$|\mathbf{C}_{\SL^\epsilon(3,q)}(h)|=(q-\epsilon)^2$. Moreover, as
$\{\lambda,\lambda^3,\lambda^5\}=\{a\lambda,a\lambda^3,a\lambda^5\}$
if and only if $a=1$, $\mathbf{C}_{\SL^\epsilon(3,q)}({h})$ is the
full pre-image of $\mathbf{C}_{\PSL^\epsilon(3,q)}(\overline{h})$.
We deduce that
$|\mathbf{C}_{\PSL^\epsilon(3,q)}(\overline{h})|=(q-\epsilon)^2/3$.
This is smaller than the $3$-part of $|\PSL^\epsilon(3,q)|$, and thus
Sylow $3$-subgroups of $\PSL^\epsilon(3,q)$ are not abelian,
violating the hypothesis. So $9$ cannot divide $ q-\epsilon$, as
stated.
\end{proof}

\begin{thm}\label{Z(P)abelian}
Let $p$ be an odd prime, $n\geq 4$, and $(n,p)\neq (6,3)$. Let
$G:=\SL^\epsilon(n,q)$ defined in characteristic not equal to $p$,
$S:=G/\bZ(G)=\PSL^\epsilon(n,q)$, and $P\in\Syl_p(S)$. Suppose that
$P/\bZ(P)$ is abelian. Then $p$ does not divide $|\bZ(G)|$.
\end{thm}

\begin{proof} Assume by contradiction that $p\mid
|\bZ(G)|=\gcd(n,q-\epsilon)$. Lemma \ref{lem:PSL-PSU} already shows
that $P$ is non-abelian, but we need to work harder to achieve that
$P/\bZ(P)$ is non-abelian. Let $\lambda \in \FF_{q^2}^\times$ be an
element of order $p$ and consider the $p$-element
\[
x:=\diag(\lambda,\lambda^{-1},I_{n-2})\in G.
\]
Let $V=\FF_q^n$, respectively $\FF_{q^2}^n$, denote the natural
$G$-module for $\epsilon=+$, respectively $\epsilon=-$. Fix a basis
$B=\{v_1,v_2,...,v_n\}$ of $V$, and consider the permutation $y$ on $B$ defined by
\[
y:=\{v_1\mapsto v_2,v_2\mapsto v_3,...,v_{p-1}\mapsto v_p,
v_p\mapsto v_1,v_i\mapsto v_i \text{ for } p<i\leq n\},
\]
which is well-defined as $p\leq n$. Note that, as $p>2$, we have
$y\in G$ and $\mathrm{ord}(y)=p$. Direct calculation shows that
$$[x,y]=\diag(\lambda^{-1},\lambda^2,\lambda^{-1},I_{n-3})=:s.$$

Suppose that the $p$-part of $q-\epsilon$ is $p^a$ and let $C$ be
the (unique) cyclic subgroup of order $p^a$ of $\FF_{q^2}^\times$.
As $y$ permutes the diagonal matrices in $G$ with diagonal entries
in $C$, one can form the corresponding semidirect product that is then a
$p$-group. It follows that $x$ and $y$ both belong to a Sylow
$p$-subgroup, say $\widehat{{P}}$, of $G$. We deduce that
$s=[x,y]\in \widehat{{P}}'$, which implies that $s\bZ(G)\in P'$,
where $P\in\Syl_p(S)$ is the image of $\widehat{{P}}$ under the
natural projection $\SL^\epsilon\rightarrow \PSL^\epsilon$.

We will show that $s\bZ(G)$ does not belong to $\bZ(P)$, which is
enough to conclude that $P/\bZ(P)$ is not abelian.

Let $\widetilde{G}:=\GL^\epsilon(n,q)$. We have
\begin{equation*}
\mathbf{C}_{\widetilde{G}}(s) = \begin{cases}
\GL^\epsilon (3,q)\times \GL^\epsilon(n-3,q) &\text{ if } p=3\\
\GL^\epsilon(1,q)\times \GL^\epsilon(2,q)\times \GL^\epsilon(n-3,q)
&\text{ if } p>3.
\end{cases}
\end{equation*}
It is easy to see that $|S:\mathbf{C}_S(s\bZ(G))|=|G:\mathbf{C}_G(s)
|=|\widetilde{G}:\mathbf{C}_{\widetilde{G}}(s)|$. Hence,
\begin{equation*}
|S:\mathbf{C}_S(s\bZ(G))| = \begin{cases}
\dfrac{|\GL^\epsilon(n,q)|}{|\GL^\epsilon(3,q)| |\GL^\epsilon(n-3,q)|} &\text{ if } p=3\\
\dfrac{|\GL^\epsilon(n,q)|}{|\GL^\epsilon(1,q)|
|\GL^\epsilon(2,q)||\GL^\epsilon(n-3,q)|} &\text{ if } p>3.
\end{cases}
\end{equation*}
It follows that, if $\ell$ is the defining characteristic of $S$,
then
\begin{equation*}
|S:\mathbf{C}_S(s\bZ(G))|_{\ell'} = \begin{cases}
\dfrac{(q^n-\epsilon^n)(q^{n-1}-\epsilon^{n-1})(q^{n-2}-\epsilon^{n-2})}{(q-\epsilon)(q^2-1)(q^3-\epsilon^3)}
&\text{ if } p=3\\
\dfrac{(q^n-\epsilon^n)(q^{n-1}-\epsilon^{n-1})(q^{n-2}-\epsilon^{n-2})}{(q-\epsilon)^2(q^2-1)}
&\text{ if } p>3.
\end{cases}
\end{equation*}
Using the condition $p\mid \gcd(n,q-\epsilon)$ and the assumption
$(n,p)\neq (6,3)$, we see that this is divisible by $p$. It follows
that $s\bZ(G)$ does not belong to $\bZ(P)$, and this finishes the
proof.
\end{proof}

\begin{lem}\label{casep3}
Let $S=\PSL^\epsilon(n,q)$ with $n\geq 4$. If $3$ divides
$q-\epsilon$, then $d_3(S)\leq \frac{1}{3}$. In particular, if $3$
divides $q-\epsilon$ and $3 \in \pi$, then $d_{\pi}(S)\leq
\frac{1}{3}$.
\end{lem}

\begin{proof}
Assume, to the contrary, that $d_3(S)>1/3$. Then $d_3(P)>1/3$, and
thus $|P'|\leq 3$ by Theorem \ref{firstpart}. The proof of Theorem
\ref{Z(P)abelian} shows that $P'$ contains two elements $s\bZ(G)$
and $t\bZ(G)$, where
$s=\diag(\lambda^{-1},\lambda^2,\lambda^{-1},I_{n-3})$ and
$t=\diag(1,\lambda^{-1},\lambda^2,\lambda^{-1},I_{n-4})$. Obviously
these elements generate a group of order greater than $3$, a
contradiction.
\end{proof}

\begin{lem}\label{lem:PSL3}
Let $S=\PSL^{\epsilon}(3,q)$ and $\pi$ a set of odd primes with
$3\in \pi$. Then $\dd_{\pi}(S)\leq \frac{1}{3}$.
\end{lem}

\begin{proof}
If $3$ does not divide $q-\epsilon$, then the result follows by
Proposition \ref{quasisimple}. We therefore assume that $3$ divides
$q-\epsilon$. In particular, $3$ divides $q^2+\epsilon q+1$. Denote
$t:=\frac{(q-\epsilon)_3}{3}$. We have
$$|S|_{3}=\frac{((q-\epsilon)^2(q+\epsilon)(q^2+\epsilon q+1))_{3}}{3}\geq (q-\epsilon)_{3}^2=9t^{2}.$$
On the other hand, counting the number of conjugacy classes of
$3$-elements in $\PSL^{\epsilon}(3,q)$ (see for example \cite{SF73})
we have $k_3(S)={(t^2+t+2)}/{2}\leq 2 t^2.$ Therefore,
$$\dd_{\pi}(S)\leq \dd_{3}(S)=\frac{k_3(S)}{|S|_{3}}
\leq \frac{ 2 t^2}{ 9 t^2}<\frac{1}{3},$$ as wanted.
\end{proof}

\begin{pro}\label{prop:PSLPSU}
Theorem \ref{teoB} holds for $S=\PSL^{\epsilon}(n,q)$ with $n\geq 3$
and $\pi=\{p,s\}$ with $p<s$ be odd primes such that $q$ is not
divisible by neither $p$ nor $s$.
\end{pro}

\begin{proof}
The result follows by Theorem \ref{quasisimple} in the case
$\gcd(n,q-\epsilon)_{\pi}=1$. So assume that
$\gcd(n,q-\epsilon)_{\pi}>1$, so that there exists $r \in \pi$ such
that $r$ divides $\gcd(n,q-\epsilon)$. The case $n=3$ is then done
by Lemma \ref{lem:PSL3}. So we assume furthermore that $n\geq 4$.

Let $R \in \Syl_{r}(S)$. We have that $R/\bZ(R)$ is abelian by
hypothesis. This and the condition $r$ divides $\gcd(n,q-\epsilon)$
contradict Theorem \ref{Z(P)abelian} if $r\geq 5$. The remaining
case $r=3$ is handled by Lemma \ref{casep3}.
\end{proof}

We have completed the proof of Theorem \ref{teoB}, by combining
Theorems \ref{prop:An}, \ref{prop:sporadic}, \ref{thm:defining
char}, \ref{prop:ABC} and Propositions \ref{prop:E6} and
\ref{prop:PSLPSU}.

As mentioned before, Theorem \ref{thm:main} follows from Theorem \ref{teoB} and Theorem \ref{4.10} together with Theorem \ref{3.4}.


\section{Examples and further discussion}\label{sec:examples}

In this section, we present examples showing that the  converses of
both statements of Theorem \ref{thm:main} are false and the bounds
are generically sharp.

Consider the converse of the first part of Theorem \ref{thm:main}.
Assume first that $2 \in \pi$ and $3 \not \in \pi$. If $G$ is the
direct product of $\Sigma_{4}$ and an abelian group, then
$\dd_{\pi}(G)=\frac{1}{6}$. Now, let $\pi$ have size at least $2$
and $p>2$. Let $P$ be a finite $p$-group with $|P'| = p$. Let $C$ be
the cyclic group which is the direct product of the groups $C_q$
where $q$ runs over all primes in $\pi$ except for $p$. Let $T$ be
the elementary abelian $2$-group of rank $|\pi|-1$. Let $G = P
\times (C : T)$ where $C:T = \prod_{p \not= q \in \pi} (C_{q} :
C_{2})$. In this case \begin{align*}\dd_{\pi}(G) &\leq \Big(
\frac{p^{2}+p-1}{p^{3}} \Big) \Big( \prod_{p \not= q \in \pi}
\frac{q+1}{2q} \Big) \\
&\leq \Big( \frac{p^{2}+p-1}{p^{3}} \Big) \cdot {\Big(
\frac{p+1}{2p} \Big)}^{|\pi|-1}\\
&\leq \Big( \frac{p^{2}+p-1}{p^{3}} \Big) \cdot \Big( \frac{p+1}{2p}
\Big).\end{align*} Since $p \geq 3$, this is less than
$\frac{5}{6p}$, so the converse of the first statement is false.

Consider the converse of the second statement. Assume first that $2
\in \pi$ and $3 \not \in \pi$. If $G$ is the direct product of
$\Al_{4}$ and an abelian group, then $\dd_{\pi}(G)=\frac{1}{6}$.
Now, let $p\not=2$ and let $|\pi| \geq 3$. Let $C = \prod_{q \in
\pi} C_q$. Let $T = C_{p-1} \times {(C_{2})}^{|\pi|-1}$ and set $G =
C:T$. Then \[\dd_{\pi}(G) = \frac{2}{p} \cdot \prod_{p \not= q \in
\pi} \frac{q+1}{2q}.\] Since $|\pi| \geq 3$, $q \geq p+2$ and all
primes $q$ in $\pi$ are odd, we get $$\dd_{\pi}(G) \leq \Big(
\frac{2}{p} \Big) \cdot \Big( \frac{(p+2)+1}{2(p+2)} \Big) \cdot
\Big( \frac{(p+4)+1}{2(p+4)} \Big) \leq \frac{24}{35p}.$$ Thus the
converse of the second statement of Theorem \ref{thm:main} is also false.

The inequality $\dd_{\pi}(G) > \frac{p^{2}+p-1}{p^{3}}$ in the
second statement of Theorem \ref{thm:main} is sharp for every set of primes
$\pi$. Take $G$ to be the direct product of a finite non-abelian
$p$-group $P$ such that $P/\bZ(P)$ is isomorphic to $C_{p} \times
C_{p}$ with an abelian group. In this case $\dd_{\pi}(G) =
\frac{p^{2}+p-1}{p^{3}}$ and $G$ does not contain an abelian Hall
$\pi$-subgroup.

Let us consider now the inequality $\dd_{\pi}(G) > 1/p$ of the first
part. This condition is best possible when $p=2$ and $3 \in \pi$,
for if $G$ is a direct product of $\Sigma_3$ and an abelian group,
then $\dd_{\pi}(G) = 1/2$ and $G$ does not contain a nilpotent Hall
$\pi$-subgroup. However the bound is certainly not best possible
when $p$ is odd. In fact, following our proofs closely, it can be
seen that in such case, the group $G$ still possesses a nilpotent
Hall $\pi$-subgroup even when $d_\pi(G)=1/p$.

Now let $p$ be odd. We will show that in certain cases the
inequality $\dd_{\pi}(G)
> 1/2p$ does not imply that $G$ has a nilpotent Hall $\pi$-subgroup.
To see this let $\pi = \{ p, q \}$ where $q = 2p+1$; that is, $p$ is
a Sophie Germain prime. Let $G$ be the direct product of
$C_{q}:C_{p}$ and an abelian group. Elementary character theory
gives $k_{\pi}(C_{q}:C_{p}) = p + (q-1)/p$. Thus \[\dd_{\pi}(G) =
\dd_{\pi}(C_{q}:C_{p})=\frac{1}{2p+1} \left(1 +
\frac{2}{p}\right),\] which is strictly larger than $1/2p$.

The last example naturally raises the following question: for $\pi$
a set of \emph{odd} primes, what is the exact (lower) bound for
$d_\pi(G)$ to ensure the existence of a nilpotent Hall
$\pi$-subgroup in $G$? This seems nontrivial to us at the time of
this writing.

Let $G$ be a finite group and let $p$ be the smallest
 prime dividing $|G|$. If $n(p)$ denotes the smallest prime larger than $p$ and
\[
\Pro(G)> \frac{n(p)+p^2-1}{p^2n(p)}=:f(p),
\]
then $|G'|\leq p$ and thus $G$ is
nilpotent by Theorem \ref{n(p)} and Lemma \ref{Gprima}. Note that $f(p)\leq 1/p$ and
equality occurs if and only if $p=2$.

Now let $\pi$ be a set of primes and $p$ be the smallest member in
$\pi$. It is perhaps true that if $d_\pi(G)>f(p)$ then $G$ possesses
a nilpotent Hall $\pi$-subgroup, but this would require significant
more effort, especially on the part of simple groups of Lie type in
characteristic not belong to $\pi$. We have decided to work with the
bound $1/p$ instead in order to make our arguments flowing smoothly.
We certainly do not claim that $f(p)$ is the (conjectural) best
possible bound for $d_\pi(G)$ to ensure the existence of a nilpotent
Hall $\pi$-subgroup in $G$, and thus the question we just raised
above remains open.

%
%
%


\end{document}